\documentclass[oneside]{amsart}
\usepackage{hyperref}

\usepackage[dvips]{pict2e}
\usepackage{amsmath,amsthm}
\usepackage{comment}

\usepackage[a4paper]{geometry}

\theoremstyle{definition}
\newtheorem{nul}{}[section]
\newtheorem{dfn}[nul]{Definition}

\newtheorem{rmk}[nul]{Remark}
\newtheorem{cnstr}[nul]{Construction}

\newtheorem{exm}[nul]{Example}

\newtheorem{qst}{Question}
\newtheorem*{dfn*}{Definition}
\newtheorem*{axm*}{Axiom}
\newtheorem*{ntn*}{Notation}
\newtheorem*{exm*}{Example}
\newtheorem*{exr*}{Exercise}
\newtheorem*{int*}{Intuition}
\newtheorem*{qst*}{Question}
\newtheorem*{cnstr*}{Construction}

\theoremstyle{plain}

\newtheorem{thm}[nul]{Theorem}
\newtheorem{prop}[nul]{Proposition}
\newtheorem{lem}[nul]{Lemma}

\newtheorem{cor}{Corollary}[nul]
\newtheorem*{thm*}{Theorem}
\newtheorem*{prop*}{Proposition}
\newtheorem*{cor*}{Corollary}
\newtheorem*{lem*}{Lemma}
\newtheorem*{cnj*}{Conjecture}

\newtheorem*{ThomThm*}{Theorem 4.10 of \cite{ThomRing}}

\numberwithin{equation}{nul}

    \newtheoremstyle{TheoremNum}
			  {2mm}{2mm}              
        {\itshape}                      
        {}                              
        {\bfseries}                     
        {.}                             
        { }                             
        {\thmname{#1}\thmnote{ \bfseries #3}}
    \theoremstyle{TheoremNum}

\DeclareMathOperator{\smsh}{\wedge}

\usepackage{tikz}
\usetikzlibrary{matrix,arrows,decorations}
\usepackage{tikz-cd}

\begin{document}

\title{Nilpotence in $\mathbb{E}_n$-algebras}
\author{Jeremy Hahn}


\begin{abstract}
Nilpotence in the homotopy of $\mathbb{E}_\infty$-ring spectra is detected by the classical $H\mathbb{Z}$-Hurewicz homomorphism.  Inspired by questions of Mathew, Noel, and Naumann, we investigate the extent to which this criterion holds in the homotopy of $\mathbb{E}_n$-ring spectra.  For all odd primes $p$ and all chromatic heights $h$, we use the Cohen-Moore-Neisendorfer theorem to construct examples of $K(h)$-local, $\mathbb{E}_{2n-1}$-algebras with non-nilpotent $p^n$-torsion.  We exploit the interaction of the Bousfield-Kuhn functor on odd spheres and Rezk's logarithm to show that our bound is sharp at height $1$, and remark on the situation at height $2$. 
\end{abstract}
\maketitle

\tableofcontents

\vbadness=10000

\section{Introduction}

Fix a prime number $p$.  A landmark result of homotopy theory is Nishida's proof \cite{Nishida} that all positive degree elements in the stable homotopy groups of spheres are smash nilpotent.  Nishida saw fit to divide his paper into two parts, containing two rather different arguments:

\begin{enumerate}
\item The argument in Part I applies only to $p$-torsion elements $x \in \pi_*(\mathbb{S})$, but gives strong bounds on the minimum $n$ such that $x^n=0$. 
\item The argument in Part II applies to all $p$-power torsion $x \in \pi_*(\mathbb{S}),$ but gives weaker bounds on the exponent of nilpotence.
\end{enumerate}

With modern understanding, Nishida's two arguments generalize to two different theorems:

\begin{thm}[Hopkins-Mahowald]
Suppose that $R$ is an $\mathbb{E}_2$-ring spectrum and $x \in \pi_*(R)$ satisfies $px=0$.  Then $x$ is nilpotent if and only if its Hurewicz image in $H_*(R;\mathbb{F}_p)$ is nilpotent.
\end{thm}

\begin{thm}[May nilpotence conjecture]
Suppose that $R$ is an $\mathbb{E}_\infty$-ring spectrum and $x \in \pi_*(R)$ satisfies $p^kx=0$ for some integer $k$.  Then $x$ is nilpotent if and only if its Hurewicz image in $H_*(R;\mathbb{F}_p)$ is nilpotent.
\end{thm}

These theorems stand in contrast to the general nilpotence technology of Devinatz, Hopkins, and Smith \cite{NilpotenceI}, which does not require the ring spectrum $R$ to be even $\mathbb{E}_2$-structured, but does require the use of the more sophisticated homology theory $MU$ of complex bordism.  Proofs of both theorems may be found in the recent work of Mathew, Noel, and Naumann \cite{MNN}, including the first recorded proof of the May nilpotence conjecture.  A natural question arises:

\begin{qst} \label{MainQuestion}
Does the May nilpotence conjecture hold for $\mathbb{E}_n$-ring spectra with $2 \le n < \infty$, or is the $\mathbb{E}_\infty$-hypothesis necessary? 
\end{qst}

Following \cite{MNN}, we break Question \ref{MainQuestion} into a series of questions at each chromatic height.  To be more precise, in Section \ref{UniversalConstruction} we recall the construction of an $\mathbb{E}_n$-ring spectrum $R_{n,k}$, depending on integers $n,k>0$, that is the \textit{universal} $\mathbb{E}_n$-algebra in which $p^k=0$.  The following is then a simple consequence of the work of Devinatz, Hopkins, and Smith:

\begin{dfn}
We say that a ring spectrum $R$ \textit{satisfies the $k$th May nilpotence criterion} if, for all $x \in \pi_*(R)$, the two conditions
\begin{enumerate}
\item $x$ has nilpotent Hurewicz image in $H_*(R;\mathbb{F}_p)$, and
\item $p^k x =0$
\end{enumerate} 
together imply that $x$ is nilpotent in $\pi_*(R)$.
\end{dfn}

\begin{thm} \label{UniversalMay}
The following are equivalent:
\begin{enumerate}
\item Every $\mathbb{E}_n$-ring spectrum satisfies the $k$th May nilpotence criterion.
\item For all integers $h \ge 1$, $K(h) \smsh R_{n,k} \simeq 0$.
\end{enumerate}
\end{thm}

The theorem of Hopkins and Mahowald may thus be restated as the observation that $R_{2,1}$ is $K(h)$-acyclic for all $h$.  In yet other words, the free $\mathbb{E}_2$-algebra with $p=0$ is $K(h)$-acyclic for all heights $h$.

In remarks following \cite[Proposition $4.19$]{MNN}, Mathew, Noel, and Naumann state that they do not know whether $K(h) \smsh R_{2,2}$ is null for any $h$.  In the best of all worlds, one might hope that $K(h) \smsh R_{2,k} \simeq 0$ for all $h$ and $k-$this would imply the May nilpotence conjecture for $\mathbb{E}_2$-algebras.  The main observation of this work is that we do not live in such a world:

\begin{thm} \label{MainCounterexampleTheorem}
Suppose $p$ is an odd prime.  Then, for all integers $n,h \ge 1$, $$K(h) \smsh R_{2n-1,n} \not\simeq 0.$$  In other words, $L_{K(h)}R_{2n-1,n}$ is an example of a $K(h)$-local, $\mathbb{E}_{2n-1}$-algebra with a non-nilpotent but $p^n$-torsion element.
\end{thm}

Our proof of Theorem \ref{MainCounterexampleTheorem}, undertaken in Section \ref{CounterexampleSection}, relies on deep work of Cohen, Moore, and Neisendorfer \cite{CMNI,CMNII,Neisendorfer} that only holds at odd primes.  Specifically, using their work, we construct a highly non-trivial map 
$$R_{2n-1,n} \longrightarrow MU/p,$$
with image the complex cobordism spectrum mod $p$.  The composite
$$
\begin{tikzcd}
MU \arrow{r}{\text{unit}} & MU \smsh R_{2n-1,n} \arrow{r} & MU/p
\end{tikzcd}
$$
is the canonical map expressing $MU/p$ as the cofiber of the multiplication by $p$ map on $MU$, and it follows that $R_{2n-1,n}$ cannot be $K(h)$-acyclic.  

\begin{rmk} If $k_1>k_2 \gg n$, our arguments with the Cohen-Moore-Neisendorfer theorem imply that the natural map
$$\mathbb{S}/k_1 \longrightarrow \mathbb{S}/k_2$$
factors through an $\mathbb{E}_n$-algebra.  This induces curious power operations on the homotopy of Moore spectra (see Remark \ref{MoorePow}).
\end{rmk}

Having shown that the $n$th May nilpotence criterion does not hold in $\mathbb{E}_{2n-1}$-algebras, we next ask whether it holds in $\mathbb{E}_{2n}$-algebras.  When $n=1$, this is true, and is again the theorem of Hopkins and Mahowald.  In general, we are able to show at least the following fact:

\begin{thm} \label{introheight1thm}
There is no $K(1)$-local, $\mathbb{E}_{2n}$-algebra with non-nilpotent $p^n$-torsion.  In other words,
$$K(1) \smsh R_{2n,n} \simeq 0.$$
\end{thm}

We prove Theorem \ref{introheight1thm} at all primes, and not merely the odd ones.  This raises the following question, to which we do not know the answer:

\begin{qst}
For any $n>1$, is there a non-trivial, $K(1)$-local, $\mathbb{E}_{2n-1}$-algebra $R$ such that $2^n=0$ in $\pi_*(R)$?
\end{qst}

Let us discuss a bit about our strategy for proving Theorem \ref{introheight1thm}.  The May nilpotence conjecture for $\mathbb{E}_\infty$-ring spectra is proven by use of power operations in Morava $E$-theory.  One wants to know what shadows of those power operations are still available in $\mathbb{E}_{\ell}$-ring spectra, and this is of course connected to study of the map 
\begin{equation} \label{equhi}
\Omega^{\ell} S^{\ell} \longrightarrow QS^0.
\end{equation}

Applying the Bousfield-Kuhn functor $\Phi_1$ \cite{BousfieldFunctor} to (\ref{equhi}) allows us to easily reason about at least one power operation, namely the Rezk logarithm \cite{RezkLog}.  In Section \ref{height1oddSection}, we use the logarithm to prove Theorem \ref{introheight1thm} at odd primes.  A minor modification of the argument would also prove Theorem \ref{introheight1thm} at the prime $2$, but, with an eye toward providing a second set of techniques, we give a different proof for $p=2$ in Section \ref{height1even}.

Our proof at $p=2$ is based on the fact that $\text{Config}(2,\mathbb{R}^\ell)_{h\Sigma_2} \simeq \mathbb{RP}^{\ell-1}$, which has computable Morava $E$-theory (at any height).  In particular, we are able to easily analyze the height two situation in Section \ref{height2}:

\begin{thm} \label{Investigation}
Let $E$ denote a Morava $E$-theory at height $2$, to be defined more explicitly in Section \ref{height2}.  Suppose that $R$ is a $K(2)$-local, $\mathbb{E}_4$-$E$-algebra and that $4=0$ in $\pi_0(R)$.  Then 
$$2u_1^2=u_1^6=0 \in \pi_0(R).$$
If $R$ is furthermore an $\mathbb{E}_{12}$-$E$-algebra, then $1=0$ in $\pi_0(R)$ and so $R \simeq 0$.
\end{thm}

Based on Theorem \ref{Investigation}, it seems as though the free $\mathbb{E}_4$-algebra with $4=0$ might not be $K(2)$-acyclic (though it must be $K(1)$-acyclic).  We are left with the following two questions:

\begin{qst} \label{SearchForCounterexamples}
Does there exists a non-trivial, $K(2)$-local, $\mathbb{E}_4$-algebra $R$ such that $p^2=0$ in $\pi_0(R)$?
\end{qst}

\begin{qst}
Let $f_p(k,h)$ denote the smallest $n$ such that the free $\mathbb{E}_n$-algebra with $p^k=0$ is $K(h)$-acyclic.  What is the asymptotic behavior of $f_p(k,h)$?
\end{qst}

It seems probable that the techniques of this paper, combined with those of \cite{HahnBousfield}, yield coarse lower bounds for $f_p(k,h)$--we leave this for future work.  It would be fascinating if Question \ref{SearchForCounterexamples} were tied to chromatic generalizations of the Cohen-Moore-Neisendorfer theorem, perhaps along the lines of those hinted at in \cite{Guozhen}.

\begin{rmk}
In the body of the paper to follow, all spectra are implicitly $p$-completed.
\end{rmk}

\textbf{Acknowledgments:} The author thanks Akhil Mathew, Jun Hou Fung, Eric Peterson, Allen Yuan, and especially his PhD advisor Mike Hopkins for helpful conversations.  This work was supported by an NSF Graduate Fellowship under Grant DGE-1144152.

\section{Universal torsion \texorpdfstring{$\mathbb{E}_{n}$}{En}-algebras} \label{UniversalConstruction}

Let $A$ be any $\mathbb{E}_\infty$-ring spectrum.  The ($\infty$)-category of $A$-modules is then symmetric monoidal, and so the subcategory consisting of the unit and its automorphisms is an infinite loop space denoted $BGL_1(A)$.  As defined in \cite{ThomDef}, the Thom spectrum of a map of spaces $$f:X \rightarrow BGL_1(A)$$ is the colimit of the composite functor $$f:X \rightarrow BGL_1(A) \subset A\text{-}\textbf{Modules}.$$  As was first noted by Lewis \cite[IX]{Lewis} and further developed in \cite{ThomRing}, the Thom spectrum $\text{Thom}(f)$ is naturally an $\mathbb{E}_n$-$A$-algebra whenever $f$ is an $n$-fold loop map.  The space $GL_1(A) \simeq \Omega BGL_1(A)$ may alternatively be described as a pullback
$$
\begin{tikzcd}
GL_1(A) \arrow{r} \arrow{d} & \Omega^{\infty} A \arrow{d} \\
\pi_0(A)^{\times} \arrow{r} & \pi_0(A).
\end{tikzcd}
$$
From this description it is clear that $\pi_1(BGL_1(A)) \cong \pi_0(A)^{\times},$ and $\pi_*(BGL_1(A)) \cong \pi_{*-1}(A)$ for $*>1$.

\begin{dfn}
Suppose we are given some $x \in \pi_0(A)$ that is one less than a unit.  Then, for each $n \ge 0$, there is an $n$-fold loop map 
$$\widetilde{1+x}:\Omega^n S^{n+1} \rightarrow BGL_1(A)$$
adjoint to the map $1+x:S^1 \rightarrow BGL_1(A)$.  We use the symbol $R^{A}_{n,x}$ to denote the resulting Thom $\mathbb{E}_n$-$A$-algebra.
\end{dfn}

\begin{ThomThm*}  For any $\mathbb{E}_n$-$A$-algebra $B$, the space of $\mathbb{E}_n$-$A$-algebra maps from $R^A_{n,k}$ to $B$ is equivalent to the space of homotopies between the two $A$-module maps
$$0:A \rightarrow B \textit{\hspace{0.5cm} and \hspace{0.5cm}} x:A \rightarrow B.$$
In particular, an $\mathbb{E}_n$-$A$-algebra map $R^A_{n,x} \rightarrow B$ exists if and only if $x=0$ in $\pi_0(B)$.  In the case $n=0$, we interpret an $\mathbb{E}_0$-$A$-algebra to mean an $A$-module $B$ equipped with a unit homomorphism $A \rightarrow B$.  One defines the map $x:A \rightarrow B$ as the precomposition of the unit with the multiplication by $x$ map on $A$. 
\end{ThomThm*}

As a result of the above theorem, we will sometimes call $R^A_{n,x}$ the universal or free $\mathbb{E}_n$-$A$-algebra in which $x=0$.  We will be particularly interested in the case that $A=\mathbb{S}$ is the ($p$-complete) sphere spectrum.  Notice that, for each integer $k \ge 1$, there is a unit $1+p^k \in \pi_0(\mathbb{S})$.

\begin{dfn}
We use the shorthand $R_{n,k}$ to denote $R^{\mathbb{S}}_{n,p^k}$, the free $\mathbb{E}_n$-ring spectrum in which $p^k=0$.
\end{dfn}

\begin{rmk} \label{plocalRmk}
Later, when $p$ is an odd prime, we will make use of the decomposition of the group of $p$-adic units $\mathbb{Z}^{\times}_p$ as a product $\mu_{p-1} \times (1+p\mathbb{Z}_p)^{\times}$.  There is a subgroup $H$ of units that are $1$ more than a multiple of $p$ (as opposed to some other root of unity plus a multiple of $p$).  By choosing a full symmetric monoidal subcategory of $GL_1(\mathbb{S})$, one can make an infinite loop map $BH \rightarrow BGL_1(\mathbb{S})$.  There is a factorization
$$\widetilde{1+p^k}:\Omega^n S^{n+1} \rightarrow BH \rightarrow BGL_1(\mathbb{S}).$$
Since $BH$ is a $p$-complete space, $\widetilde{1+p^k}$ also factors through the $p$-completion of $\Omega^n S^{n+1}$.
\end{rmk}

\begin{thm}
The following are equivalent:
\begin{enumerate}
\item Every $\mathbb{E}_n$-ring spectrum satisfies the $k$th May nilpotence criterion.
\item For all integers $h \ge 1$, $L_{K(h)}(R_{n,k})$ is the trivial $\mathbb{E}_n$-algebra.
\end{enumerate}
\end{thm}

\begin{proof}
First, let us show that $(1) \Rightarrow (2)$.  For each integer $h>0$, let $A_h$ denote the $\mathbb{E}_n$-algebra $L_{K(h)} R_{n,k}$.   Since $A_h$ is $K(h)$-local, it is $L_h$-local.  Since $L_h$ is smashing, we deduce that $$H\mathbb{Z} \smsh A_h \simeq \left(L_{h} H\mathbb{Z}\right) \smsh A_h \simeq H\mathbb{Q} \smsh A_h.$$
As $p^k=0$ in $\pi_0(A_h)$, $H\mathbb{Q} \smsh A_h \simeq 0$.  Using the May nilpotence criterion, we conclude that $1$ is nilpotent in $A_h$, so $A_h \simeq 0$.

Next, we show that $(2) \Rightarrow (1)$.  Suppose that $A$ is an arbitrary $\mathbb{E}_n$-algebra and $x \in \pi_*(A)$ satisfies $p^k x = 0$.  We must show that, if $x$ has nilpotent Hurewicz image in $H_*(A)$, then it is nilpotent in $\pi_*(A)$.  By the work of Devinatz, Hopkins, and Smith \cite{NilpotenceI, NilpotenceII}, it suffices to show that $x$ has nilpotent image in $\pi_*(L_{K(h)} A)$ for every $h \ge 1$.  For each $h$, this is equivalent to showing that the $\mathbb{E}_n$-algebra $L_{K(h)}A[x^{-1}]$ is trivial.  The universal property of $R_{n,k}$ guarantees that there exists a map $R_{n,k} \rightarrow L_{K(h)}A[x^{-1}]$, which must in turn factor through $L_{K(h)}R_{n,k} \simeq 0$.  Only the trivial $\mathbb{E}_n$-algebra receives a map from the trivial $\mathbb{E}_n$-algebra.
\end{proof}

\section{Counterexamples to May nilpotence} \label{CounterexampleSection}

Assume $p>2$ is odd.  For any $n \ge 1$, we will prove that the spectrum $R_{2n-1,n}$ of Section \ref{UniversalConstruction} is not $K(h)$-acyclic for any height $h \ge 1$.  In particular, this shows that the May nilpotence criterion fails for $p^n$-torsion classes in $\mathbb{E}_{2n-1}$-algebras.  The restriction to odd primes allows us to utilize the deep work of Cohen, Moore, and Neisendorfer \cite{CMNI, CMNII, Neisendorfer}:

\begin{thm}[Cohen-Moore-Neisendorfer] \label{CMNTheorem}
Suppose $p$ is an odd prime and $n \ge 1$.  Then, after $p$-localization, the $p$th power map $$p:\Omega^{2} S^{2n+1} \rightarrow \Omega^{2} S^{2n+1}$$ factors through the double suspension $E^{2}:S^{2n-1} \rightarrow \Omega^{2} S^{2n+1}$.
\end{thm}

\begin{cor}
After $p$-localization, the $(p^n)$th power map
$$p^{n}:\Omega^{2n} S^{2n+1} \rightarrow \Omega^{2n} S^{2n+1}$$
factors through $E^{2n}:S^1 \rightarrow S^{2n+1}$. 
\end{cor}

Using Theorem \ref{CMNTheorem} and the following lemma, we will give a quick proof that the free $\mathbb{E}_{2n-2}$-algebra with $p^n=0$ is not $K(h)$-acyclic.  As a byproduct of the proof we obtain a curious sort of power operation on Moore spectra.  Finally, we end the section with a slightly more complicated argument that handles $\mathbb{E}_{2n-1}$-algebras.

\begin{lem} \label{NumberLemma}
For any integers $a,n>0$, the map $\widetilde{1+p^{n}}:\Omega^{a}S^{a+1} \rightarrow BGL_1(\mathbb{S})$ factors as a composition
$$
\begin{tikzcd}
\Omega^{a}S^{a+1} \arrow{r}{p^{n-1}} &  \Omega^{a}S^{a+1} \arrow{r}{\widetilde{1+p\alpha}} & BGL_1(\mathbb{S}),
\end{tikzcd}
$$
where $\alpha \in \mathbb{Z}_p^{\times}$ is a $p$-adic unit such that $(1+p\alpha)^{p^{n-1}} = 1+p^{n}$.
\end{lem}

\begin{proof} 
It is more generally the case that, given any $p$-adic unit $b$ and integer $k>0$, there exists a $p$-adic unit $c$ such that $(1+bp^k)=(1+pc)^{p^{k-1}}$.  One can see this by induction on $k$, the case $k=1$ being trivial.  For $k>1$, it suffices by induction to find a unit $x$ such that $(1+xp^{k-1})^p = 1+bp^k$.  We may rewrite this equation as
$$1+\binom{p}{1} xp^{k-1} + \binom{p}{2} x^2 p^{2(k-1)} + \cdots + \binom{p}{p} x^p p^{p(k-1)} = 1 + bp^k,$$
which (since $p>2$ and $k>1$) reduces to 
$$x-b = p f(x)$$
for some polynomial $f(x)$.  We can solve this equation by Hensel's lemma.
\end{proof}

\begin{thm} \label{MooreProjection}
The natural map of Moore spectra $\mathbb{S}/p^{n} \rightarrow \mathbb{S}/p$ factors through $R_{2n-2,n}$
\end{thm}

\begin{proof}
Recall the sequence
$$
\begin{tikzcd}
\Omega^{2n-2}S^{2n-1} \arrow{r}{p^{n-1}} &  \Omega^{2n-2}S^{2n-1} \arrow{r}{\widetilde{1+p\alpha}} & BH \arrow{r} & BGL_1(\mathbb{S}),
\end{tikzcd}
$$
where $BH$ is the $p$-complete subspace of $BGL_1(\mathbb{S})$ constructed in Remark \ref{plocalRmk}.  Since $BH$ is in particular $p$-local, we obtain from the Cohen-Moore-Neisendorfer theorem an (implicitly $p$-localized) diagram
$$
\begin{tikzcd}
S^1 \arrow{r}{E^{2n-2}} & \Omega^{2n-2} S^{2n-1} \arrow{d} \arrow{r}{p^{n-1}} & \Omega^{2n-2}S^{2n-1} \arrow{r}{\widetilde{1+p\alpha}} & BH. \\
& S^1 \arrow[swap]{ur}{E^{2n-2}}
\end{tikzcd}
$$
The Thom spectrum of the composite 
$$
\begin{tikzcd}
S^1 \arrow{r}{E^{2n-2}} & \Omega^{2n-2}S^{2n-1} \arrow{r}{\widetilde{1+p\alpha}} & BH \arrow{r} & BGL_1(\mathbb{S}).
\end{tikzcd}
$$
is $\mathbb{S}/(p\alpha) \simeq \mathbb{S}/p$, whereas the Thom spectrum of
$$
\begin{tikzcd}
S^1 \arrow{r}{E^{2n-2}} & \Omega^{2n-2} S^{2n-1} \arrow{r}{\widetilde{1+p^n}} & BGL_1(\mathbb{S})
\end{tikzcd}
$$
is $\mathbb{S}/p^n$.
\end{proof}

For any height $h \ge 1$, the natural map $\mathbb{S} \rightarrow \mathbb{S}/p$ remains non-trivial after smashing with $K(h)$.  There is thus no unital map from a trivial algebra into $L_{K(h)} \mathbb{S}/p$, and it follows that:

\begin{cor}
For any height $h \ge 1$, the $K(h)$-localization of $R_{2n-2,n}$ is not trivial.  Thus, the $n$th May nilpotence criterion does not hold in the homotopy of $K(h)$-local, $\mathbb{E}_{2n-2}$-algebras.
\end{cor}

\begin{rmk} \label{MoorePow}
The proof of Theorem \ref{MooreProjection} highlights a sort of power operation connecting the homotopy groups of Moore spectra.  If $k_1 > k_2 \gg n$, the map $\mathbb{S}/p^{k_1} \rightarrow \mathbb{S}/p^{k_2}$ will factor through the free $\mathbb{E}_n$-algebra with $p^{k_1}=0$.  For any $\mathbb{E}_n$-ring spectrum $E$, the power operations in $E \smsh R_{n,k_1}$ thus induce operations $E_0(\mathbb{S}/p^{k_1}) \rightarrow E_0(\mathbb{S}/p^{k_2})$.  This should be compared with \cite{LawsonDavis}, who prove that the entire tower $\{\mathbb{S}/p^k\}_{k \ge 1}$ is pro-$\mathbb{E}_\infty$.
\end{rmk}

The remainder of the section is devoted to the proof of the more comprehensive theorem:

\begin{thm}
For any height $h \ge 1$, the $K(h)$-localization of $R_{2n-1,n}$ is non-trivial.  Thus, the $n$th May nilpotence criterion does not hold in the homotopy of $K(h)$-local $\mathbb{E}_{2n-1}$-algebras.
\end{thm}

\begin{proof}
We use the well-known result of Serre that, after $p$-localization,
$$S^{2n-1} \times \Omega S^{4n-1} \simeq \Omega S^{2n}.$$
The equivalence is induced by the product of two maps \cite{Jelena}:
\begin{enumerate}
\item The suspension map $E:S^{2n-1} \rightarrow \Omega S^{2n},$ and
\item Loops of the Whitehead square of the identity $[\iota,\iota]:S^{4n-1} \rightarrow S^{2n}$.
\end{enumerate}
We conclude, from Lemma \ref{NumberLemma} and the Cohen-Moore-Neisendorfer Theorem, that the map
$$\widetilde{1+p^n}:\Omega^{2n-1} S^{2n} \longrightarrow BGL_1(\mathbb{S})$$
factors as a composite
$$
\begin{tikzcd} 
\Omega^{2n-2} S^{2n-1} \times \Omega^{2n-1} S^{4n-1} \arrow{d} \arrow{r}{p^{n-1}} & \Omega^{2n-2} S^{2n-1} \times \Omega^{2n-1} S^{4n-1} \arrow{r}{\simeq} & \Omega^{2n-1} S^{2n} \arrow{d}{\widetilde{1+p\alpha}} \\
S^1 \times \Omega S^{2n+1} \arrow[swap]{ur}{(E^{2n-2},\Omega E^{2n-2})} \arrow[swap]{rr}{f} && BGL_1(\mathbb{S}).
\end{tikzcd}
$$
It follows that there is a unital map from $R_{2n-1,n}$ to the Thom spectrum of the map 
$$f:S^1 \times \Omega S^{2n+1} \rightarrow BGL_1(\mathbb{S}).$$  Thus, to show that $R_{2n-1,n}$ is not $K(h)$-acyclic it suffices to show that $$K(h) \smsh \text{Thom}(f) \not\simeq 0.$$

The map $f$ is a product of two maps $$f \simeq \left( f_1:S^1 \rightarrow BGL_1(\mathbb{S}) \right) \times \left(f_2:\Omega S^{2n+1} \rightarrow BGL_1(\mathbb{S})\right),$$ which can be described as follows:
\begin{enumerate}
\item The map $f_1$ is $1+p\alpha:S^1 \rightarrow BGL_1(\mathbb{S})$, and so has associated Thom spectrum $\mathbb{S}/\alpha p \simeq \mathbb{S}/p$.
\item The map $f_2$ is the loop map adjoint to \textit{some} map of pointed spaces $$\gamma:S^{2n} \rightarrow BGL_1(\mathbb{S}).$$  We will not need to identify $\gamma \in \pi_{2n-1}(\mathbb{S})$.  By \cite[4.10]{ThomRing}, $\text{Thom}(f_2)$ is the universal $\mathbb{A}_\infty$-ring spectrum in which $\gamma=0$.
\end{enumerate}

Now, $\text{Thom}(f) \simeq \text{Thom}(f_1) \smsh \text{Thom}(f_2) \simeq \mathbb{S}/p \smsh \text{Thom}(f_2)$.  Since the Moore spectrum $\mathbb{S}/p$ is not $K(h)$-acyclic, we are reduced to showing that $\text{Thom}(f_2)$ is not $K(h)$-acyclic.  For this, it suffices to show that $\gamma=0$ in \textit{some} non-trivial $K(h)$-local, $\mathbb{A}_\infty$-ring spectrum.  We may use $L_{K(h)}MU$, since $\gamma \in \pi_{2n-1}(\mathbb{S})$ lives in odd degree.

\end{proof}

\section{The height one case at an odd prime} \label{height1oddSection}

Assume $p>2$ is an odd prime.  Recall that the free $\mathbb{E}_{2n}$-algebra with $p^n=0$, denoted $R_{2n,n}$, is the Thom spectrum of the $(2n)$-fold loop map $$\widetilde{1+p^n}:\Omega^{2n} S^{2n+1} \rightarrow BGL_1(\mathbb{S}).$$  In this section, we will prove:

\begin{thm} \label{HeightOneOddThm}
The $K(1)$-localization $L_{K(1)} R_{2n,n} \simeq 0$.  Thus, the May nilpotence criterion holds for $p^n$ torsion elements in $K(1)$-local $\mathbb{E}_{2n}$-ring spectra.
\end{thm}

To begin the proof, recall that the theory of orientations (see, e.g., \cite[3.7]{ThomRing}) implies that the composite
$$
\begin{tikzcd} [column sep=huge]
\Omega^{2n} S^{2n+1} \arrow{r}{\widetilde{1+p^n}} & BGL_1(\mathbb{S}) \arrow{r}{\text{$BGL_1$(unit)}} & BGL_1(R_{2n,n})
\end{tikzcd}
$$
is null.  Taking loops once, and abusing notation by using $\widetilde{1+p^n}$ to also refer to the infinite loop map adjoint to $1+p^n \in \pi_0(GL_1(\mathbb{S}))$, we obtain the null composite
$$
\begin{tikzcd} [column sep=huge]
\Omega^{2n+1} S^{2n+1} \arrow{r} & \Omega^{\infty} \mathbb{S} \arrow{r}{\widetilde{1+p^n}} & GL_1(\mathbb{S}) \arrow{r}{\text{$GL_1$(unit)}} & GL_1(R_{2n,n}).
\end{tikzcd}
$$
To this sequence, we apply the $K(1)$-local Bousfield Kuhn functor $\Phi_1$ \cite{BousfieldFunctor}, obtaining
$$
\begin{tikzcd} [column sep = huge]
\Phi_1(\Omega^{2n+1} S^{2n+1}) \arrow{r} & L_{K(1)} \mathbb{S} \arrow{r}{\text{log}(1+p^n)} & L_{K(1)}\mathbb{S} \arrow{r}{L_{K(1)}\text{unit}} &L_{K(1)} R_{2n,n},
\end{tikzcd}
$$
where $\text{log}$ denotes the $K(1)$-local Rezk logarithm defined in \cite[\S 3]{RezkLog}.  A formula for the logarithm is given as Theorem $1.9$ in \cite{RezkLog}:

\begin{thm*}[Rezk]
The operation $$\text{log}:\pi_0(L_{K(1)} \mathbb{S})^{\times} \rightarrow \pi_0(L_{K(1)} \mathbb{S})$$ is given by the formula
$$\text{log}(x)=\sum_{k=1}^{\infty} (-1)^k \frac{p^{k-1}}{k} \left( \frac{\theta(x)}{x^p} \right)^k,$$
where 
$$\theta:\pi_0(L_{K(1)} \mathbb{S}) \rightarrow \pi_0(L_{K(1)} \mathbb{S})$$
is an operation such that $x \mapsto x^p+p\theta(x)$ is a ring homomorphism.
\end{thm*}

\begin{cor}
The relation 
$$\text{log}(1+p^n) \equiv p^{n-1} \text{ modulo $p^n$}$$
holds in $\pi_0(L_{K(1)} \mathbb{S})$.  Therefore, the image of $\text{log}(1+p^n)$ in $\pi_0(L_{K(1)} R_{2n,n})$ is equal to $p^{n-1}$.
\end{cor}

\begin{proof}
The relation $$p\theta(1+p^n)+(1+p^n)^p = 1+p^n$$ tells us that 
$$p\theta(1+p^n) \equiv p^n \text{ modulo $p^{n+1}$},$$
and since $\pi_0(L_{K(1)} \mathbb{S})$ has no torsion for $p$ odd we learn that $\theta(1+p^n) \equiv p^{n-1}$ modulo $p^n$.  The formula for Rezk's logarithm then degenerates to tell us that
$$\text{log}(1+p^n) \equiv \frac{p^{n-1}}{(1+p^n)^p} \equiv p^{n-1},$$
modulo $p^n$.	
\end{proof}

\begin{proof}[Proof of Theorem \ref{HeightOneOddThm}]
From the above, it follows that the composite
$$
\begin{tikzcd}
\Phi_1(\Omega^{2n+1}S^{2n+1}) \arrow{r} & L_{K(1)} \mathbb{S} \arrow{r}{p^{n-1}} & L_{K(1)} R_{2n,n}
\end{tikzcd}
$$
is null.  According to \cite[3.1]{DavisBook}, $$\Phi_1(\Omega^{2n+1} S^{2n+1}) \simeq L_{K(1)} \mathbb{S}^{-1}/p^n,$$ and the map
$\Phi_1(\Omega^{2n+1}S^{2n+1}) \longrightarrow L_{K(1)} \mathbb{S}$ is the $K(1)$ localization of the first map in the cofiber sequence
$$\mathbb{S}^{-1}/p^n \longrightarrow \mathbb{S} \stackrel{p^n}{\longrightarrow} \mathbb{S} \longrightarrow \mathbb{S}/p^n.$$

It follows that, in $\pi_0(L_{K(1)} R_{2n,n})$, $p^{n-1}$ is a multiple of $p^n$, and so $p^{n-1}=0$.  Inducting on the value of $n$ completes the proof of Theorem \ref{HeightOneOddThm}.
\end{proof}

\begin{rmk}
A naive adaptation of this argument to the $K(2)$-local setting fails, as the $K(2)$-local logarithm (for a height $2$ Morava $E$-theory) of $1+p^n$ is $0$ rather than $p^{n-1}$.  Roughly speaking, this is because the $\widetilde{1+p^n}$ map interacts exclusively with the $\text{Im $J$}$ part of the group of units of the sphere.  Rather, one expects to need an inductive argument: applying the $K(1)$-local logarithm implies that some power of $u_1$ is zero, and then one should apply the $K(2)$-local logarithm to $1$ more than that power of $u_1$.
\end{rmk}

\section{Power operations and transfers at \texorpdfstring{$p=2$}{p=2}}

We now turn our attention to the prime $p=2$, where we prove an analogue of Theorem \ref{HeightOneOddThm} by somewhat different means.  With an eye toward investigations at height $2$, which we undertake in Section \ref{height2}, we begin with some fairly general remarks.

Let $A$ denote an $\mathbb{E}_\infty$-ring spectrum, and suppose we are given some $x \in \pi_0(A)$ that is one less than a unit.  In this section, we develop a basic tool for analyzing the group $\pi_0(R^A_{n,x})$, where $R^A_{n,x}$ is the free $\mathbb{E}_n$-$A$-algebra with $x=0$ constructed in Section \ref{UniversalConstruction}.  In Section \ref{height1even}, we will use this tool to analyze the $K(1)$-local homotopy type of $R_{n,k} \simeq R^{\mathbb{S}}_{n,2^k}$.

The key idea is that, since $x$ vanishes in $\pi_0(R^A_{n,x})$, so must various power operations evaluated on $x$.  Let us review the theory of power operations in $\mathbb{E}_\infty$-ring spectra.  The original reference is \cite{HinfBook} and a good reference for the specific facts we need is \cite[\S 7,8]{RezkLog}.

\begin{cnstr}
If $X$ is an $\mathbb{E}_\infty$-space, then there is associated to every (unpointed) map of spaces $* \rightarrow X$ a canonical structure map
$$*_{h\Sigma_2} \simeq B\Sigma_2 \simeq \mathbb{RP}^{\infty} \rightarrow X.$$  To each element $x \in \pi_0(X)$ is thus associated a homotopy class $\bar{x} \in [B\Sigma_2,X]$.  If $$X \simeq \Omega^{\infty} E$$ happens to be a grouplike $\mathbb{E}_\infty$-space, then we may represent an $x \in \pi_0(X)$ by an infinite loop map $\Omega^{\infty} \mathbb{S} \rightarrow \Omega^{\infty}E$.  Precomposition with the natural inclusion $B\Sigma_2 \rightarrow \Omega^{\infty} \mathbb{S}$ recovers $\bar{x}$.

If $A$ is an $\mathbb{E}_\infty$-ring spectrum, then $\Omega^{\infty} A$ has two natural $\mathbb{E}_\infty$-structures, one additive and the other multiplicative.  We will use $P(x)$ to denote the element of $A^0(B\Sigma_2)$ given by the multiplicative $\mathbb{E}_\infty$-structure, and $tr(x)$ to denote the element in $A^0(B\Sigma_2)$ given by the additive structure.  There are relations 
\begin{enumerate}
\item $P(x_1x_2)=P(x_1)P(x_2),$
\item $tr(x_1 + x_2)=tr(x_1)+tr(x_2)$, and
\item $P(x_1+x_2)=P(x_1)+P(x_2)+tr(x_1x_2).$
\end{enumerate}
If $1+x$ happens to be a unit in $\pi_0(A)$, then we may view $1+x$ as a class in $\pi_0(GL_1(A))$.  The infinite loop space structure on $GL_1(A)$ then gives rise to a composite, well-defined up to homotopy,
$$
\begin{tikzcd} [column sep=large]
B\Sigma_2 \arrow{r}{\overline{1+x}} &  GL_1(A) \arrow[hook]{r} & \Omega^{\infty} A,
\end{tikzcd}
$$
and this is the class $P(1+x)$.
\end{cnstr}

Now, let $A$ denote an $\mathbb{E}_\infty$-ring spectrum and $x$ a class in $\pi_0(A)$ that is one less than a unit.  Recall that the free $\mathbb{E}_{n}$-$A$-algebra with $x=0$, denoted $R^A_{n,x}$, is the Thom spectrum of the $n$-fold loop map $$\widetilde{1+x}:\Omega^{n} S^{n+1} \rightarrow BGL_1(A).$$  The theory of orientations (see, e.g., \cite[3.7]{ThomRing}) implies that the composite
$$
\begin{tikzcd} [column sep=huge]
\Omega^{n} S^{n+1} \arrow{r}{\widetilde{1+x}} & BGL_1(A) \arrow{r}{\text{$BGL_1$(unit)}} & BGL_1(R^A_{n,x})
\end{tikzcd}
$$
is null.  Taking loops once, we ponder the following diagram
$$
\begin{tikzcd}[column sep = huge]
\mathbb{RP}^n \arrow{r} \arrow{d}{\simeq} & \mathbb{RP}^{\infty} \arrow{d}{\simeq} \\
\text{Config}(2,\mathbb{R}^{n+1})_{h\Sigma_2} \arrow{r} \arrow{d} & B\Sigma_2 \arrow{d} \\
\Omega^{n+1} S^{n+1}  \arrow{r} & \Omega^{\infty} \mathbb{S} \arrow{r}{\Omega^{\infty}(1+x)} & GL_1(A) \arrow[hook]{d} \arrow{r}{GL_1(\text{unit})} & GL_1(R^A_{n,x}) \arrow[hook]{d} \\
&  & \Omega^{\infty} A \arrow{r}{\Omega^{\infty}(\text{unit})} & \Omega^{\infty} R^A_{n,x}
\end{tikzcd}
$$
and conclude:

\begin{lem} \label{RawPowOp}
The composition 
$$
\begin{tikzcd}[column sep=large]
\Sigma^{\infty}_+ \mathbb{RP}^n \arrow{r} & \Sigma^{\infty}_+ \mathbb{RP}^{\infty} \arrow{r}{P(1+x)} & A \arrow{r}{unit} & R^A_{n,x}
\end{tikzcd}
$$
is equal to the unit $1$ in the ring $(R^A_{n,x})^0 \left(\mathbb{RP}^{n} \right)$.
\end{lem}

\begin{dfn} 
We use the notation $\mathbb{RP}_{n+1}^{\infty}$ to denote the Thom spectrum of the direct sum of $(n+1)$ copies of the canonical bundle over $B\Sigma_2$.  There is a well-known cofiber sequence
$$
\begin{tikzcd}[column sep=huge]
\Sigma_+^{\infty} \mathbb{RP}^{n} \arrow{r} & \Sigma_+^{\infty} \mathbb{RP}^{\infty} \arrow{r}{0\text{-section}} & \mathbb{RP}^{\infty}_{n+1}.
\end{tikzcd}
$$
\end{dfn}

\begin{prop} \label{BetterPowOp}
Let $A$ denote an $\mathbb{E}_\infty$-ring, $n>0$ an integer, and $x \in \pi_0(A)$ a class that is one less than a unit.  Then there is a commuting diagram of spectra
$$
\begin{tikzcd} [column sep = huge]
A \arrow{r}{\text{unit}} & R^A_{n,x} \\ 
\Sigma^{\infty}_+ \mathbb{RP}^{\infty} \arrow{u}{P(x)} \arrow{r}{0\text{-section}} & \mathbb{RP}^{\infty}_{n+1}. \arrow{u}
\end{tikzcd}
$$
\end{prop}

\begin{proof}
According to the above cofiber sequence, it is only necessary to prove that the composite
$$
\begin{tikzcd}
\Sigma^{\infty}_+ \mathbb{RP}^n \arrow{r} & \Sigma^{\infty}_+ \mathbb{RP}^{\infty} \arrow{r}{P(x)} \arrow{r} & A \arrow{r} & R^A_{n,x} 
\end{tikzcd}
$$
is null.  By Lemma \ref{RawPowOp} we know that the composite is null when the map $P(x)$ is replaced by the map $P(1+x)-1$.  As it turns out, the composition of $P(x)$ with the unit $A \rightarrow R^{A}_{n,x}$ is homotopic to the composition of $P(1+x)-1$ with the unit $A \rightarrow R^{A}_{n,x}$.  This is because of the relation
$$P(1+x)=P(1)+P(x)+tr(x)= P(1)+P(x)+x tr(1),$$
together with the fact that $x$ maps to $0$ in $\pi_0(R^{A}_{n,x})$.  
\end{proof}

Soon, we will set $A$ to be a Morava $E$-theory.  In such cases, we can extract concrete algebraic statements from Proposition \ref{BetterPowOp}.

\begin{cnstr}
Suppose that $k$ is a perfect field of characteristic $2$ and that $\mathbb{G}_0$ is a formal group of finite height $0<h<\infty$ over $k$.  From this data, we may form \cite[\S 6]{HopkinsMiller} a complex-oriented Morava $E$-theory with
$$E_0 \cong \pi_0(E) \cong W(k)[[u_1,u_2,\cdots,u_{h-1}]].$$
The homotopy groups $\pi_*(E)$ are $2$-periodic and concentrated in even degrees.

As explained in \cite[\S 5]{HKR}, the Gysin sequence of the fibration
$$\mathbb{RP}^{\infty} \longrightarrow BS^1 \stackrel{2}{\longrightarrow} BS^1$$
allows one to construct an isomorphism
$$E^0(\mathbb{RP}^{\infty}) \cong E_0[z]/[2](z),$$
where $z$ is the Chern class of the complexification of the canonical real line bundle.  The Weierstrass Preparation Theorem guarantees that $E^0(\mathbb{RP}^{\infty})$ is a free and finitely generated $E_0$-module.

Similarly, the fibration
$$\mathbb{RP}^{2n-1} \rightarrow \mathbb{CP}^{n} \rightarrow \mathbb{CP}^{n}$$
gives that $E^0(\mathbb{RP}^{2n-1}) \cong E_0[z]/([2](z),z^n)$, while $E^1(\mathbb{RP}^{2n-1}) \cong E_0$.  Comparison of Atiyah-Hizerburch spectral sequences shows that
$$E^0(\mathbb{RP}^{2n}) \cong E^0(\mathbb{RP}^{2n+1}),$$
but $E^1(\mathbb{RP}^{2n}) \cong 0$.  For us, the upshot is that the map in reduced cohomology
$$\widetilde{E}^0(\mathbb{RP}^{\infty}_{2n}) \rightarrow \widetilde{E}^0(\mathbb{RP}^{\infty})$$
is given by the inclusion of free $E_0$-modules
$$z^{n+1} \left( \frac{E_0[z]}{[2](z)} \right) \hookrightarrow z \left( \frac{E_0[z]}{[2](z)} \right),$$
and furthermore $E^1(\mathbb{RP}^{\infty}_{2n}) \cong 0$.

Given a spectrum $X$, the \textit{completed} $E$-homology of $X$ is by definition 
$$E^{\wedge}_0(X) \cong \pi_0(L_{K(h)} E \smsh X).$$
In the case that $E^0(X)$ is finitely generated, free, and concentrated in even degrees, we may identify $E^{\wedge}_0(X)$ with the $E_0$-linear dual of $X$.
\end{cnstr}

We obtain the following computational corollary of Proposition \ref{BetterPowOp}: 
\begin{cor} \label{PowOpExplicit}
Let $E$ denote the Morava $E$-theory corresponding to a formal group $\mathbb{G}_0$ of height $0<h<\infty$ over a perfect field $k$ of characteristic $2$, and suppose that $x \in \pi_0(E)$ is in the maximal ideal.  Then there is a diagram of $E_0$-modules

$$
\begin{tikzcd}
E_0 \arrow{r} & \pi_0(L_{K(h)} R^E_{n,x}) \\
\left(z \frac{E_0[z]}{[2](z)}\right)^{\vee} \arrow{r} \arrow{u}{\overline{P}(x)} & \left(z^{n+1} \frac{E_0[z]}{[2](z)}\right)^{\vee} \arrow{u}
\end{tikzcd}
$$

Here, $M^\vee$ denotes the $E_0$-linear dual of a free and finitely generated $E_0$-module $M$.  By $\overline{P}(x)$ we mean the composite
$$
\begin{tikzcd} [column sep= huge]
\pi_0\left( L_{K(h)} E \smsh \Sigma^{\infty} \mathbb{RP}^{\infty} \right) \arrow{r}& \pi_0\left( L_{K(h)} E \smsh \Sigma^{\infty}_+ \mathbb{RP}^{\infty} \right) \arrow{r}{\pi_0 L_{K(h)}P(x)} & \pi_0(E).
\end{tikzcd}
$$
\end{cor}

\section{The height one case at \texorpdfstring{$p=2$}{p=2}} \label{height1even}

Here, we use Corollary \ref{PowOpExplicit} to prove the $2$-primary version of Theorem \ref{HeightOneOddThm}.  Let $R_{2n,n}$ denote the free $\mathbb{E}_{2n}$-ring spectrum in which $2^n=0$, as constructed in Section \ref{UniversalConstruction}.

\begin{thm} \label{HeightOneEvenThm}
The $K(1)$-localization $L_{K(1)} R_{2n,n} \simeq 0$.  Thus, the May nilpotence criterion holds for $2^n$-torsion elements in $K(1)$-local, $\mathbb{E}_{2n}$-ring spectra.
\end{thm}

Let $E$ denote the $2$-completion of $KU$.  This $E$ is a height one Morava $E$-theory with $E_0 \cong \mathbb{Z}_2$ and $[2](x)=2x-x^2 \in E_0[[x]]$.  Since $K(1) \smsh E$ is a wedge of suspensions of $K(1)$, we see that $K(1) \smsh E \smsh R_{2n,n}$ is null if and only if $K(1) \smsh R_{2n,n}$ is null.  Thus, to prove Theorem \ref{HeightOneEvenThm} it suffices to prove that 
$$L_{K(1)} \left( E \smsh R_{2n,n} \right) \simeq L_{K(1)} R^E_{2n,2^n}$$
is the trivial $\mathbb{E}_{2n}$-$E$-algebra.

By Corollary \ref{PowOpExplicit}, there is a diagram of $\mathbb{Z}_2$-modules

\begin{equation*} \label{eqndiagram}
\begin{tikzcd}[column sep = large] 
\mathbb{Z}_2 \arrow{r}{\text{unit}} & \pi_0\left(L_{K(1)} R^E_{2n,2^n}\right) \\
\left(z\frac{\mathbb{Z}_2[z]}{2z-z^2}\right)^{\vee} \arrow{u}{\bar{P}(2^n)} \arrow{r} & \left(z^{n+1} \frac{\mathbb{Z}_2[z]}{2z-z^2} \right)^{\vee}. \arrow{u}
\end{tikzcd} \tag{$\star$}
\end{equation*}

There is a natural identification of $\mathbb{Z}_2$-modules $\left(z\frac{\mathbb{Z}_2[z]}{2z-z^2}\right) \cong \mathbb{Z}_2$, in which the element $z^k$ is identified with $2^k$.  We view $$\left(z\frac{\mathbb{Z}_2[z]}{2z-z^2}\right)^{\vee}$$
as the free $\mathbb{Z}_2$-module on the single dual basis element $\delta_{z}$, while
$$\left(z^{n+1} \frac{\mathbb{Z}_2[z]}{2z-z^2} \right)^{\vee}$$
is the free $\mathbb{Z}_2$-module with basis $\delta_{z^{n+1}}$.  The map
$$\left(z\frac{\mathbb{Z}_2[z]}{2z-z^2}\right)^{\vee}  \longrightarrow \left(z^{n+1} \frac{\mathbb{Z}_2[z]}{2z-z^2} \right)^{\vee}$$
sends $\delta_{z}$ to $2^{n} \delta_{z^{n+1}}$

\begin{lem}
The map $\bar{P}(2^n)$ sends $\delta_{z}$ to $2^{n-1}$ plus some multiple of $2^n$.
\end{lem}

\begin{proof}
From the expression 
$$P(2^n)=P(2^{n-1}) +P(2^{n-1})+tr(2^{n-1} \cdot 2^{n-1})=2 P(2^{n-1})+2^{2n-2} tr(1),$$ together with the base case $P(1)=1$, we inductively learn that
$$P(2^n) \equiv 2^{n-1} tr(1) \text{ modulo $2^n$}.$$
The lemma then follows from the fact (see, e.g., \cite{HKR}), that
$$tr(1)=\frac{[2](z)}{z} = 2-z.$$
\end{proof}

Using the lemma, we may rewrite (\ref{eqndiagram}) as follows:

$$
\begin{tikzcd}[column sep = huge]
\mathbb{Z}_2 \arrow{r}{\text{unit}} & \pi_0 \left(L_{K(1)} R^E_{2n,2^n}\right) \\ \\
\mathbb{Z}_2 \delta_z   \arrow{r}{\delta_z \mapsto 2^n \delta_{z^{n+1}}}  \arrow{uu}{\delta_z \mapsto 2^{n-1}+m2^n} & \mathbb{Z}_2 \delta_{z^{n+1}} , \arrow{uu}
\end{tikzcd}
$$

and from this we learn that $2^{n-1}$ is a multiple of $2^n$ in $\pi_0 \left(L_{K(1)} R^E_{2n,2^n}\right)$.  Since $2^n=0$ in this $\mathbb{Z}_2$-module, we learn that $2^{n-1}=0$ as well.  Inducting on $n$, we conclude that $1=0$ in $\pi_0 \left(L_{K(1)} R^E_{2n,2^n}\right)$, and so $L_{K(1)} R^E_{2n,2^n}$ is the trivial $\mathbb{E}_{2n}$-algebra.

\section{Calculations at height two} \label{height2}

In Section \ref{height1even}, we showed that any $\mathbb{E}_4$-algebra $R$ such that $4=0$ in $\pi_0(R)$ must be $K(1)$-acyclic.  In this section, we ask whether such an algebra must also be $K(2)$-acyclic.  We cannot settle the question, but we do show that an $\mathbb{E}_{12}$-structure is enough to force $R$ to be $K(2)$-acyclic.

We hope to return to other primes and other heights in future work, where we expect to produce bounds generalizing the $12$ appearing here.  Nonetheless, the analysis of this section well demonstrates the limits of current technology, and we hope the exposition inspires work on the following:

\begin{qst}
Is it possible to construct a non-trivial, $K(2)$-local $\mathbb{E}_{p^2}$-algebra with $p^2=0$?
\end{qst}

In the note \cite{RezkPower2}, Rezk uses an explicit elliptic curve to define a $K(2)$-local $\mathbb{E}_\infty$-ring $E$ with
$$\pi_0(E) \cong \mathbb{Z}_2[[a]].$$
The ring spectrum $E$ is a version of height $2$ Morava $E$-theory at the prime $2$, and receives a complex orientation from $BP$ that sends $v_1$ to a unit multiple of $a$.  The following explicit formulae are easily derived from those given in \cite{RezkPower2} (our notation differs from his in that we use $P$ to denote the full power operation rather than its projection onto the quotient by the transfer ideal):

\begin{enumerate}
\item $E^0(B\Sigma_2) \cong E^0(\mathbb{RP}^{\infty}) \cong \mathbb{Z}_2[[a]][z]/(z^4-az^2-2z)$.
\item $\text{tr}(1)=2+az-z^3$.
\item $P(a) = a^2+3z-az^2$.
\end{enumerate}

\begin{prop} \label{Workhorse}
Let $x$ denote an element in the maximal ideal of $\pi_0(E)$, and suppose that $P(x)=x^2+p_1 z + p_2 z^2 + p_3 z^3$, where $p_1,p_2,p_3 \in \mathbb{Z}_2[[a]]$.  Suppose that $R$ is a $K(2)$-local, $\mathbb{E}_4$-$E$-algebra, and note that the unit map $\pi_0(E) \rightarrow \pi_0(R)$ allows us to view $x,p_1,p_2,$ and $p_3$ as elements in $\pi_0(R)$.  Then, if $x=0$ in $\pi_0(R)$, there exist elements $r,s,t \in \pi_0(R)$ such that:
$$p_1=2s,$$
$$p_2=as+2t, \text{ and}$$
$$p_3=r+at.$$
\end{prop}

\begin{proof}
It suffices to show the proposition for $R=L_{K(2)} R^E_{4,x}$, the universal $K(2)$-local, $\mathbb{E}_4$-$E$-algebra with $x=0$.  According to Corollary \ref{PowOpExplicit}, there is a diagram of $E_0$-modules

$$
\begin{tikzcd}
E_0 \arrow{r} & \pi_0(L_{K(h)} R^E_{n,x}) \\
\left(z \frac{E_0[z]}{(z^4-az^2-2z)}\right)^{\vee} \arrow{r} \arrow{u}{\overline{P}(x)} & \left(z^{3} \frac{E_0[z]}{(z^4-az^2-2z)}\right)^{\vee} \arrow{u}
\end{tikzcd}
$$

A basis for $\left(z \frac{E_0[z]}{(z^4-az^2-2z)}\right)^{\vee}$ is given by $\delta_z,\delta_{z^2},$ and $\delta_{z^3}$, while a basis for $\left(z^{3} \frac{E_0[z]}{(z^4-az^2-2z)}\right)^{\vee}$ is given by $\delta_{z^3}, \delta_{z^4},$ and $\delta_{z^5}$.

From the relations $z^3=z^3, z^4=az^2+2z$, and $z^5=az^3+2z^2$, we learn that the lower horizontal map in the diagram is given by
$$\delta_z \mapsto 2 \delta_{z^4},$$
$$\delta_{z^2} \mapsto a \delta_{z^4} + 2 \delta_{z^5},$$
$$\delta_{z^3} \mapsto 1 \delta_{z^3} + a \delta_{z^5}.$$

Since $\overline{P}(x)$ takes $\delta_z$ to $p_1$, $\delta_{z^2}$ to $p_2$, and $\delta_{z^3}$ to $p_3$, we learn that there are elements $r,s,$ and $t$ in $\pi_0\left(L_{K(h)} R^E_{4,x}\right)$ such that
$$p_1=2s,$$
$$p_2=as+2t, \text{ and}$$
$$p_3=r+at.$$
\end{proof}

\begin{exm} Suppose $x=2$ in Proposition \ref{Workhorse}.  Then $$P(x)=P(1+1)=2+tr(1)=4+az-z^3,$$ and so $p_1=a,p_2=0,$ and $p_3=-1$.  We learn that there are elements $r,s,t \in \pi_0(R)$ such that
$$a=2s,$$
$$0=as+2t, \text{ and}$$
$$-1=r+at.$$
Since $2=0$ in $\pi_0(R)$, so does $a$.
\end{exm}

\begin{exm} \label{helpexm}
Suppose $x=a$ in Proposition \ref{Workhorse}.  Then $P(x)=a^2+3z-az^2$, and so $p_1=3,p_2=-a,$ and $p_3=0$.  We learn that there exist $r,s,t \in \pi_0(R)$ such that
$$3=2s,$$
$$-a=as+2t, \text{ and}$$
$$0=r+at.$$
In particular, $3$ is a multiple of $2$ in $\pi_0(R)$, and since $R$ is $K(2)$-local it follows that $R$ is trivial.  From the previous example, we learn that the free $\mathbb{E}_2$-algebra with $2=0$ is $K(2)$-acyclic.
\end{exm}

\begin{rmk}
The above examples are not surprising in light of the fact that the free $\mathbb{E}_2$-algebra with $2=0$ is $H\mathbb{F}_2$.
\end{rmk}

\begin{exm}
Suppose $x=4$ in Proposition \ref{Workhorse}.  Then
$$P(x)=P(2+2)=4+6tr(1) \equiv 2az+2z^3 \text{ modulo 4}$$

Since $4=0$ in $\pi_0(R)$, we learn that there are some $r,s,t \in \pi_0(R)$ such that
$$2a = 2s,$$
$$0=as + 2t, \text{ and}$$
$$2= r + a t.$$

Multiplying the second equation by $2$, we learn that $2as=0$.  Multiplying the first equation by $a$, we conclude that $2a^2=0$.  This implies in particular that $2=0$ in $a^{-1} R$, and by the previous examples $a^{-1} R$ must be trivial.  In short, some power of $a$ must be $0$ in $\pi_0(R)$, and we will soon see that $a^6=0$.
\end{exm}

\begin{exm}
Suppose $x=2a^2$ in Proposition \ref{Workhorse}, so 
$$P(x)=P(2)P(a)^2=(2+tr(1))P(a)^2 \equiv (a^4-12a)z^3+18z^2+a^5 z \text{ modulo }2a^2.$$  We learn that there are some $r,s,t \in \pi_0(R)$ such that,
$$a^5= 2s,$$
$$18=as + 2t, \text{ and}$$
$$a^4-12a = r +a t.$$
Multiplying the middle equation by $2$, and substituting the first equation, we learn that $a^6$ is divisible by $4$.  If $4=0$ in $\pi_0(R)$, it follows that $a^6=0$.
\end{exm}

\begin{exm}
Now, suppose that $R$ is a $K(2)$-local $E$-algebra in which $4=0$.  According to the previous examples, we also learn that $2a^2=0$ and $a^6=0$.  One might hope to extract further information by using Proposition \ref{Workhorse} with $x=a^6$.  Note that
$$P(a^6)=P(a)^6 \equiv  2az+a^2z^2 + 2z^3 \text{ modulo } (4,2a^2,a^6).$$
We conclude that there are $r,s,t \in \pi_0(R)$ such that
$$2a=2s,$$
$$a^2=as +2t, \text{ and}$$
$$2=r+at.$$
It does not appear possible to learn anything new about $\pi_0(R)$ from these equations.
\end{exm}

We now explore the appropriate variant of Proposition \ref{Workhorse} for $K(2)$-local $\mathbb{E}_{12}$-$E$-algebras.

\begin{prop} \label{Workhouse2}
Let $x$ denote an element in the maximal ideal of $\pi_0(E)$, and suppose that $P(x)=x^2+p_1 z + p_2 z^2 + p_3 z^3$, where $p_1,p_2,p_3 \in \mathbb{Z}_2[[a]]$.  Suppose that $R$ is a $K(2)$-local, $\mathbb{E}_{12}$-$E$-algebra, and note that the unit map $\pi_0(E) \rightarrow \pi_0(R)$ allows us to view $x,p_1,p_2,$ and $p_3$ as elements in $\pi_0(R)$.  Then, if $x=0$ in $\pi_0(R)$, there exist elements $r,s,t \in \pi_0(R)$ such that:
$$p_1=4r+2a^2 s + 8a t,$$
$$p_2=4a r + (a^3+4)s + 6a^2 t, \text{ and}$$
$$p_3=a^2 r + 4a s + (a^3+4)t.$$
\end{prop}

\begin{proof}
The proof is the obvious modification of that given for \ref{Workhorse}.  In this case, we examine the map
$$
\begin{tikzcd}
\left(z \frac{E_0[z]}{(z^4-az^2-2z)}\right)^{\vee} \arrow{r} & \left(z^{7} \frac{E_0[z]}{(z^4-az^2-2z)}\right)^{\vee}.
\end{tikzcd}
$$
We learn from the relations 
$$z^7=a^2z^3+4az^2+4z,$$ 
$$z^8=4az^3+(a^3+4)z^2+2a^2z, \text{ and}$$
$$z^9=(a^3+4)z^3+6a^2z^2+8az$$ that 
$$\delta_z \mapsto 4 \delta_{z^7}+2a^2 \delta_{z^8}+8a \delta_{z^9},$$
$$\delta_{z^2} \mapsto  4a \delta_{z^7} + (a^3+4)\delta_{z^8} + 6a^2 \delta_{z^9}, \text{ and}$$
$$\delta_{z^3} \mapsto  a^2 \delta_{z^7} + 4a \delta_{z^8} + (a^3+4)\delta_{z^9}.$$
\end{proof}

\begin{cor} Suppose that $R$ is a $K(2)$-local, $\mathbb{E}_{12}$-algebra.  Then, if $4=0$ in $\pi_0(R)$, it is also true that $1=0$ in $\pi_0(R)$.
\end{cor}

\begin{proof}
Since $R$ is an $\mathbb{E}_4$-algebra, the previous examples imply that $2a^2=0$ and $a^6=0$ in $\pi_0(R)$.
Applying Proposition \ref{Workhouse2} with $x=4$, we learn the existence of $r,s,t \in \pi_0(R)$ such that
$$2a = 4r+2a^2 s + 8 a t=0.$$
Applying Proposition \ref{Workhouse2} with $x=2a$, we learn the existence of $r',s',t' \in \pi_0(R)$ such that
$$2+a^3 = 4r'+2a^2s'+8at' = 0.$$
Applying Proposition \ref{Workhouse2} with $x=a^3+2$, we learn the existence of $r'',s'',t'' \in \pi_0(R)$ such that
$$a = 4r''+2a^2s''+8at''=0,$$
and the result follows by Example \ref{helpexm}.
\end{proof}

\bibliographystyle{alpha}
\bibliography{References}

\end{document}